\documentclass[journal]{IEEEtran}
\usepackage{amsmath,amsfonts}
\usepackage{mathrsfs}
\usepackage{algorithmic}
\usepackage{algorithm}
\usepackage{array}
\usepackage[caption=false, font=footnotesize]{subfig}

\usepackage{textcomp}
\usepackage{stfloats}
\usepackage{url}
\usepackage{verbatim}
\usepackage{graphicx}
\usepackage{cite}
\usepackage{amssymb}
\usepackage{amsthm}
\usepackage{amssymb}
\usepackage{amsmath}
\usepackage{amsthm}

\usepackage[pdftex,colorlinks,citecolor=black,linkcolor=black,urlcolor=black,bookmarks=false]{hyperref}
\newtheorem{theorem}{Theorem}
\newtheorem{theorem*}{Theorem}
\newtheorem{proposition}[theorem]{Proposition}
\newtheorem{definition}[theorem]{Definition}

\newtheorem{lemma}[theorem]{Lemma}

\theoremstyle{remark}
\newtheorem{remark}[theorem]{Remark}

\newcommand{\dive}{\mathrm{div}}
\newcommand{\grad}{\mathrm{grad}}
\newcommand{\ord}{\mathrm{ord}}
\newcommand{\F}{\mathcal{F}}
\newcommand{\calE}{\mathcal{E}}
\newcommand{\calP}{\mathcal{P}}
\newcommand{\calY}{\mathcal{Y}}
\newcommand{\EE}{\mathbb{E}}
\newcommand{\PP}{\mathbb{P}}
\newcommand{\R}{\mathbb{R}}
\newcommand{\Z}{\mathbb{Z}}

\begin{document}

\title{Nonexistence of finite-dimensional estimation algebras on closed smooth manifolds}

\author{Jiayi Kang, Andrew Salmon and Stephen Shing-Toung Yau ~\IEEEmembership{Life Fellow,~IEEE}
\thanks{ \emph{Corresponding author: Stephen Shing-Toung Yau.}}
\thanks{Jiayi Kang is with Beijing Institute of Mathematical Sciences and Applications, Huairou district, Beijing 101400, China. {(E-mail: kangjiayi@bimsa.cn).}}
\thanks{Andrew Salmon is an independent researcher. Andrew Salmon is the co-first author.
{(E-mail: asalmon@alum.mit.edu).}}
\thanks{Stephen Shing-Toung Yau is with the Beijing Institute of Mathematical Sciences and Applications, Huairou district, Beijing 101400, China and with the Department of Mathematical Sciences, Tsinghua University, Beijing 100084, China.
{(E-mail: yau@uic.edu).}}
}

\maketitle

\begin{abstract}
 Estimation algebras have been extensively studied in Euclidean space, where finite-dimensional estimation algebras form the foundation of the Kalman and Benes filters, and have contributed to the discovery of many other finite-dimensional filters.  This work extends the theory of estimation algebras to filtering problems on Riemannian manifolds in continuous time. Our main result demonstrates that, with non-constant observation functions, the estimation algebra associated with the system on closed Riemannian manifolds is infinite-dimensional.
\end{abstract}


%
\IEEEpeerreviewmaketitle

\section{Introduction}

The filtering problem, which involves estimating the state of a stochastic dynamical system under noisy observations, has been a major research focus since the 1960s when Kalman and Bucy introduced linear filtering theory for systems with Gaussian initial conditions \cite{kalman1960new,kalman1961new}. This foundational work spurred interest in nonlinear filtering. One important point of view on the continuous time filtering problem is given by the Duncan-Mortensen-Zakai (DMZ) equation \cite{duncan1967probability,mortensen1966optimal,zakai1969optimal}, which gives a stochastic partial differential equation for the probability density function of the state conditioned on observations.

In the 1970s, Brockett \cite{Brockett1980}, Clark\cite{Brockett1981}, and Mitter \cite{Mitter1979} independently advanced the field by developing finite-dimensional filters using the estimation algebra method. The estimation algebra can be unconditionally defined for continuous-time filtering systems as a Lie subalgebra of differential operators on the state space, and it enables the construction of finite-dimensional recursive filters, provided the estimation algebra is finite-dimensional. The estimation algebra approach recasts previously discovered finite-dimensional filters, such as the Kalman-Bucy filter and Bene\c{s} filters, in a common framework. Yau and his collaborators studied the structure of estimation algebras in detail during the 1990s and 2000s \cite{Yau1990,YauRasoulian1999,WuYauHu2002}. They identified a key invariant, the linear rank, the dimension of the vector subspace of the estimation algebra spanned by homogeneous linear functions, and classified all finite-dimensional estimation algebras of maximal rank, when the linear rank is equal to the dimension of the state space \cite{Yau2003}. Since the 2000s, subsequent progress has focused on non-maximal rank estimation algebras in Euclidean space. Wu and Yau classified finite-dimensional estimation algebras (FDEA) with state dimension 2 and rank 1 in 2006 \cite{WuYau2006}. Shi and Yau extended this work by proving the Mitter conjecture, which states that the observation process must be linear in the state, for state dimension 3 and rank 2 \cite{ShiYau2017}. Further, they constructed new classes of finite-dimensional filters (FDFs) with various state dimensions and ranks. Jiao and Yau recently identified a novel class of FDFs not of Yau-type, broadening the understanding of non-maximal rank estimation algebras across different state space dimensions \cite{jiao2020new}. And a time-varying extension of FDEA appeared in \cite{kang2023finite}. However, the extension of FDEA to manifold cases remains an open problem.

Suppose that $(M,g)$ is a Riemannian manifold with a Riemannian metric $g$. The filtering problem of $M$ considered here is based on the signal observation model:
\begin{equation}\label{eq main theorem system}
    \begin{split}
        dX_t &= V_0(X_t) \, dt + \sum_{\alpha=1}^l V_\alpha(X_t) \circ dW^\alpha_t \\
        dY_t &= h(X_t)\, dt + dV_t
    \end{split}
\end{equation}
where $X_t\in M$, $Y_t \in \R^m$,  the $V_0,V_1,\cdots,V_l$ are all smooth vector fields of $M$, $h:M\rightarrow \R^m$ is the smooth function, $W_t^\alpha$ is 1 dimensional standard Brownian motion on $\R$, and $V_t$ is $m$ dimensional standard Brownian motion. 

Our  main result, established in Theorem \ref{thm:non exist} of Section 4, demonstrates the nonexistence of finite-dimensional estimation algebras when the state equations are defined on a closed compact Riemannian manifold, with the stochastic component arising from Brownian motion on the manifold. We prove the following result:

\begin{theorem*}[Main Theorem]\label{Main Theorem}
    Let $M$ be a closed smooth manifold, and let $L$ be a diffusion process on $M$ with a second-order term that is non-degenerate everywhere (in the sense of Definition~\ref{def: non-degenerate diffusion process}). Let $h$ be the observation function of a signal-observation model as in Equation~\eqref{eq main theorem system} and assume that $h$ is not a constant function. Then the estimation algebra (see Definition~\ref{def: estimation algebra}) attached to the signal-observation model is infinite-dimensional.
\end{theorem*}

More generally, we show the following result which shows that in many cases when $M$ is non-compact, the estimation algebra is still infinite-dimensional, especially if $M$ has nontrivial topology.

\begin{theorem*}[Main Theorem II]\label{Generalization of Main Theorem}
    Let $M$ be a possibly non-compact manifold without boundary, and let $L$ be a diffusion process on $M$ with second order term that is non-degenerate everywhere.  Let $h$ be the observation function, and let $x$ and $y$ be distinct critical points of some component $h^i$ of $h$ on $M$ connected by a gradient flow $\gamma \colon \R \rightarrow M$, i.e. $\gamma$ satisfies the differential equation
    \[ \gamma'(t) = (\grad h)(\gamma(t)) \]
    and approaches $x$ at $-\infty$ and $y$ at $+\infty$.  Then the estimation algebra attached to the signal-observation model is infinite-dimensional.
\end{theorem*}



A few previous works have explored the equations for filtering stochastic processes on Riemannian manifolds \cite{duncan1977filtering,ng1985nonlinear,solo2009nonlinear}, although to our knowledge, the estimation algebra approach has not been explicitly considered previously. Despite being studied much less compared to their Euclidean state space counterpart, many filtering problems are most naturally phrased on Riemannian manifolds. This is especially true for Lie groups and their homogeneous spaces, which can be used to give a uniform description of configuration spaces. For example, subspace tracking has been explored using Stiefel and Grassmann manifolds \cite{tompkins2007bayesian,srivastava2004bayesian}, offering robust solutions in these domains. Additionally, the manifold of symmetric positive definite matrices has proven instrumental for covariance tracking \cite{snoussi2013particle,hauberg2013unscented}, further highlighting the importance of manifold-based filtering techniques in tackling complex real-world challenges.

In addition to our original findings, it is important to note that while our conclusion has not been previously explored within the context of filtering, related advancements in finite-dimensional representations of stochastic partial differential equations on manifolds—where solutions are represented with finite-dimensional parameters—have been made in recent years \cite{de2018lie,albeverio2019some}. Our conclusion indicates that in the case of compact Riemannian manifolds, a finite-dimensional filter similar to the Kalman filter does not exist. This suggests that developing numerical algorithms suitable for infinite-dimensional filtering is crucial for addressing filtering problems on closed Riemannian manifolds.

\section{Filtering equations on a Riemannian manifold}

\subsection{Preliminaries and notation}

Let $(M, g)$ be a compact Riemannian manifold of dimension $n$, and let $\omega$ be the canonical measure $d\omega = \sqrt{|g|} dx^1 \wedge \cdots \wedge dx^n$.  In our setup, all functions and operators will be over the real numbers.  Let $C^\infty(M) = \Gamma(M)$ be the smooth functions on $M$, which we view as the (real-valued) smooth sections of the trivial line bundle.  The space of vector fields is the global sections $\Gamma(M, TM)$ of the tangent bundle.  Any vector field defines a linear operator on smooth functions by differentiation, which we denote $X(f)$ for $X \in \Gamma(M, TM), f \in \Gamma(M)$.  The smooth differential operators $D^{\infty}(M)$ are the noncommutative algebra over $C^\infty(M)$ of operators $C^{\infty}(M) \rightarrow C^{\infty}(M)$ generated by vector fields.  The differential operators are filtered by order, denoted $\ord(D)$ for any differential operator $D$, where the differential operators of order $0$ are the smooth functions acting on smooth functions by multiplication.  Differential operators are graded-commutative in the sense that for $D_1, D_2 \in D^{\infty}(M)$, $\ord(D_1+D_2) \le \max(\ord(D_1), \ord(D_2))$, $\ord(D_1 D_2) \le \ord(D_1) + \ord(D_2)$, and $\ord(D_1 D_2 - D_2 D_1) \le \ord(D_1) + \ord(D_2) - 1$ if $D_1$ and $D_2$ do not commute.  In the course of what follows, we will also use the commutator on differential operators $[D_1, D_2] = D_1 D_2 - D_2 D_1$ which satisfies the usual property of a Poisson bracket.

The Riemannian metric $g$ induces a pairing $\langle \cdot, \cdot \rangle_g \colon \Gamma(M, TM) \times \Gamma(M, TM) \rightarrow \Gamma(M)$ between vector fields and induces an isomorphism between the tangent bundle and the cotangent bundle.  This allows us to transfer the natural $1$-form $df$ for any smooth function to a vector field to a vector field $\grad_g f \in \Gamma(M, TM)$.  For any vector field, $X$, we can characterize the divergence $\dive_g X \in \Gamma(M)$ by
\[ \int (\dive X) f d\omega = \int \langle X, (\grad f) \rangle d\omega = \int X(f) d\omega. \]
The Laplace-Beltrami operator is $\Delta_{M,g} f = \dive_g \grad_g f$, and it is a fundamental example of a smooth second-order elliptic differential operator on $M$ and is self-adjoint. We will omit the $g$ in the gradient operator or divergence operator when the underlying Riemannian structure is clear.

Let $T^{(n,m)} M$ be the vector bundle of $(r,s)$-tensors so that in particular the fiber over $x \in M$ is canonically identified with $T^{\otimes n}_x M \otimes (T^*)^{\otimes m}_x M$.  We use the notation $T^{n} M = T^{(n, 0)} M$.  The Riemannian metric extends to a pairing $T^{(n,m)}_x M \otimes T^{(n,m)}_x M \rightarrow \R$ and a pairing on sections as $\langle \cdot, \cdot \rangle \colon \Gamma(M, T^n M \otimes T^{*m} M) \times \Gamma(M, T^n M \otimes T^{*m} M) \rightarrow \Gamma(M)$.  The covariant derivative gives a linear operator $\nabla \colon \Gamma(M, T^{(n,m)}M) \rightarrow \Gamma(M, T^{(n,m+1)}M)$ that extends the operator $d$ sending a function to a $1$-form. 

\section{Background on filtering on closed manifolds}

We recall the filtering equations, following the exposition of Bain and Crisan \cite[Chapter~3]{bain2009fundamentals}.  Let $\calP(M)$ denote the set of probability measures on $M$.  Let $(\Omega, \F, \PP)$ be a probability space with filtration $(\F_t)_{t \ge 0}$ so that $\F_t$ is complete, right-continuous, and contains all $\PP$-null sets.  Let $\EE$ denote expectation with respect to the probability measure $\PP$.

\subsection{The Riemannian metric attached to a non-degenerate diffusion process}
{A general stochastic differential equation in Stratonovich formulation has the form:
\begin{equation}\label{eq system}
   dX_t =   V_0(X_t) \, dt + \sum_{\alpha=1}^l V_\alpha(X_t) \circ dW^\alpha_t, 
\end{equation}
where the $V_0,V_1,\cdots,V_l$ are all smooth vector fields of $M$, i.e. $V_0,V_1,\cdots,V_l \in \Gamma(M,TM)$ and $W_t^\alpha$ are standard 1 dimensional Brownian motion on $\R$. 
We consider diffusion processes $X_t$ generated by $L = \frac{1}{2} \sum_{i=1}^l V_\alpha^2 + V_0$, in our setting a smooth second-order elliptic differential operator.

\begin{definition}\label{def: non-degenerate diffusion process}
A non-degenerate diffusion process on the probability space $(\Omega, \F, \PP)$ is a process $M^{\phi}$ such that for any smooth function $\phi \in C^\infty(M)$, the process $M^{\phi}_t$ defined by
\[ M^{\phi}_t = \phi(X_t) - \phi(X_0) - \int_0^t L\phi(X_s) ds \]
is an $\F_t$-adapted martingale, and $L$ is a second-order differential operator with smooth coefficients such that the second-order part is everywhere non-degenerate.
\end{definition}
We remark that for $L$ an elliptic operator, the existence of diffusion processes can be established as in \cite[Section~1.3]{hsu2002stochastic}.}  Given an $L$, we can choose a 'good' Riemannian metric $g$ to give our operator some extra structure.

\begin{lemma}\label{lemma lap}
    Consider a given diffusion process $X_t$ on $M$ generated by a non-degenerate $L$. There exists a new Riemannian metric $\tilde{g}$ such that the degree 2 part of $L$ is $\frac{1}{2}\Delta_{M,\tilde{g}}$.
\end{lemma}

\begin{proof}
In local coordinates $x_1, \cdots, x_n$, we can represent the  degree 2 part as 
\[
\text{degree 2 part of}(L) =\frac{1}{2} \sum_{i,j=1}^n a_{i,j}(x)\frac{\partial^2}{\partial x_i\partial x_j}.
\]
Since $L$ is non-degenerate, we may write $A(x)=(a_{i,j}(x))_{i,j=1}^{n}\in \R^{n\times n}$ as a invertible matrix for any $x\in M$. so we can construct the inverse of $A(x)$ as $G(x)=(g_{i,j}(x))_{i,j=1}^n$. Then, we can construct a $2$-form $\tilde{g}$ as 
\[
\tilde{g} = \sum_{i,j=1}^ng_{ij}dx_idx_j.
\]
The Laplace-Beltrami operator of $\tilde{g}$ is given by
\[
\Delta_{M,\tilde{g}}  = \sum_{i,j=1}^n\frac{1}{\sqrt{\text{det}(G)}} \frac{\partial}{\partial x_j}\left(\sqrt{\text{det}(G)}a_{i,j} \frac{\partial }{\partial x_i}\right).
\]
Here $\text{det}(G)$ is the determinant of the matrix $G$. We have
\[
\frac{1}{2}\Delta_{M,\tilde{g}}  = \frac{1}{2}\sum_{i,j=1}^na_{i,j} \frac{\partial^2 f}{\partial x_i \partial x_j} + b^i \frac{\partial }{\partial x_i},
\]
where
\[
b^i = \sum_{j=1}^n\frac{1}{\sqrt{G}} \frac{\partial(\sqrt{G}a_{i,j}(x))}{\partial x_j} = \sum_{j,k=1}^na_{jk}\Gamma^i_{jk},
\]
where $\Gamma^i_{jk}$ are the Christoffel symbols of the metric $\tilde{g}$ in the local coordinates. 
\end{proof}

\subsection{The DMZ equation on Riemannian manifolds}

{As seen above, for any given non-degenerate diffusion system on a compact smooth manifold without boundary as in Equation~\eqref{eq system}, we can choose a metric $g$ of $M$ as Lemma \ref{lemma lap} such that the generator $L$ of the diffusion process can be decomposed into a Laplacian part $\frac{1}{2}\Delta_M$ and a degree $1$ vector field part $F \in \Gamma(M,TM)$, i.e.,
\begin{equation}\label{eq def of L}
    L = \frac{1}{2}\Delta_{M,g} + F.
\end{equation}}
For simplicity, we donate $\Delta_{M,g} = \Delta_{M}$ in the rest of the paper.  The aim of this section is to generalize the construction of the estimation algebra to general $M$, allowing $M$ to be non-compact, so that we recover the usual theory for $M = \R^n$.

{
In the case that $L$ is $\Delta_M / 2$, then the diffusion process $X_t$ is called Brownian motion on the Riemannian manifold $M$.}

{
\begin{remark}[Relation to continuous filtering literature in Euclidean space]
    When considering $M = \R^n$, condition \eqref{eq def of L} will allow us to write the diffusion process $X_t$ as a stochastic integral. For arbitrary manifolds, we may consider this condition as giving a stochastic integral condition on $X_t$. We will sometimes find it convenient to consider the Stratonovich integrals rather than It\^o integrals in order to use the chain rule. At the same time, on manifolds, these two types of stochastic integrals can still be freely converted between each other, which can be easily seen as in \cite[Chapter~1]{hsu2002stochastic} by taking an isometric embedding into Euclidean space and taking smooth extensions of functions on the manifold.
\end{remark}}

The diffusion process $X_t$ gives the equations of motion for the state process.  We will also consider an observation process which satisfies its own evolution equation.  To simplify the exposition, let $h \colon M \rightarrow \R^m$ be a smooth function for compact $M$ or at most linear growth for $M=\R^n$, which in our setting guarantees that certain boundedness conditions (e.g. Novikov's condition) automatically hold.  We let $h^i \colon M \rightarrow \R$ be the $i$th coordinate function of $h$.

Let $W_t$ be an $\F_t$-adapted standard Brownian motion on $\R^m$ and $Y_t$ the process satisfying the evolution equation
\begin{equation}\label{eq observation}
    Y_t = Y_0 + \int_0^t h(X_s) dt + W_t.
\end{equation}
The process $Y_t$ determines an augmented filtration $\calY_t \subset \F_t$.  In this setting, there is an optional $\calY_t$-adapted probability-valued measurable process $\pi_t$ that solves the filtering problem, so that for any Borel-measurable function $\phi \colon M \rightarrow \R$,
\[ \pi_t \phi = \EE\left[ \phi(X_t) | \calY_t \right]. \]

\begin{theorem}[Theorem 3.24 in \cite{bain2009fundamentals}, or in \cite{gyongy1997stochastic}]
    Consider a filtering problem on a compact Riemannian manifold, where the system equation satisfies \eqref{eq system} and the observation equation satisfies \eqref{eq observation}. Its unnormalized density function satisfies the following DMZ equation:
    \begin{equation}
\begin{split}
\rho_t(\varphi) = \pi_0(\varphi) + \int_0^t \rho_s(A \varphi) ds + \sum_{i=1}^m\int_0^t \rho_s(\varphi h^i) dY_s^i,
\end{split}
\end{equation}
which holds $\tilde{\PP}$-a.s., for all times $t \ge 0$ for any smooth test function $\varphi$.
\end{theorem}

\begin{definition}[$H^k(M)$ on Riemannian Manifolds]
Let $(M, g)$ be a smooth closed Riemannian manifold of dimension $n$. For $k \in \mathbb{N}$, we define the reflexive Sobolev space $H^k(M)$ as:
\[
\begin{split}
H^k(M) = \Bigg\{f \in L^2(M) : &\nabla^j f \in L^2(M,{T^*}^{j}M) \text{in weak sense,} \\ & \text{ for all } 
 1\leq j \leq k\Bigg\}    
\end{split}
\]
where $\nabla^j f$ denotes the $j$-th covariant derivative of $f$ with respect to the Levi-Civita connection associated with the metric $g$. The Sobolev norm for $f \in H^k(M)$ is defined as:
\[
\|f\|_{H^k(M)} = \left(\sum_{j=0}^n \|\nabla^j f\|_{L^2(M)}^2\right)^{1/2}
\]
Moreover, $H^k(M)$ is a Hilbert space with the inner product:
\[
\langle f, g \rangle_{H^n(M)} = \sum_{j=0}^n \int_M \langle \nabla^j f, \nabla^j g \rangle \, d\omega
\]
where $d\omega$ is the canonical measure and $\langle \cdot, \cdot \rangle$ is the inner product induced by the metric $g$ on tensor fields.
\end{definition}

More importantly, we have the following theorem.

\begin{theorem}[Theorem 3.1 in \cite{gyongy1997stochastic}]
Assume that the conditional law $\pi_0$ is absolutely continuous with respect to the canonical measure $\omega$, and that $p_0 := d\pi_0/d\omega$ belongs to $C^{\infty}(M)$ for some integer $m > 0$. Then for every $t > 0$ the conditional distribution $\pi_t$, given the observations $Y_s$ for $0 \leq s \leq t$ has a density with respect to $w$. Moreover, $p = (p(t,x)) \in C^1([0,\infty), H^{m+1}(M))$ and
\[
p(t,x) = \sigma(t,x) \left(\int_M \sigma(t,x)dw\right)^{-1},
\]
where $\sigma$ is the unique generalized solution of the problem,
\begin{equation}\label{eq DMZ Ito form}
\begin{aligned}
d\sigma(t, x) &= L^* \sigma(t, x) dt + \sum_{i=1}^m L_i \sigma(t, x) dY^i_t, \\
\sigma(0, x) &= \pi_0,
\end{aligned}
\end{equation}
where $L^*$ is adjoint operator of $L$ and $L_i$ is the zero degree
differential operators of multiplication by $h^i$ for any $1\le i\le m$.
\end{theorem}

Next we will transform the equation \eqref{eq DMZ Ito form} from Ito form to Stranovich form as follows,
\begin{equation}\label{eq DMZ Stra form}
\begin{aligned}
d\sigma(t, x) &= L_0 \sigma(t, x) dt + \sum_{i=1}^m L_i\sigma(t, x)\circ dY^i_t, \\
\sigma(0, x) &= \pi_0
\end{aligned}
\end{equation}
where $L_0:= L^*-\frac{1}{2}h^\top h$.

We now recall an exponential transform originally introduced by Davis \cite{davis1980multiplicative} adapted to our setting, which has been extensively analyzed in the case that $M$ is Euclidean space.  Define a new unnormalized density by
\[ u(t, x) = \exp\left( - \sum_{i=1}^m h^i(x) Y^i_t \right) \sigma(t, x). \]
\begin{proposition}
    The density $u \colon \R \times M \rightarrow \R$ satisfies the following partial differential equation.
    \begin{equation}
        \begin{aligned}
            \frac{\partial u}{\partial t} &= L_0 u + \sum_{i=1}^m Y^i_t [L_0, L_i] u + \frac{1}{2} \sum_{i,j=1}^m Y^i_t Y^j_t [[L_0, L_i], L_j] u, \\
            u(0, x) &= \pi_0.
        \end{aligned}
    \end{equation}
\end{proposition}

\begin{proof}
   We begin by applying the chain rule for Stratonovich integrals to equation \eqref{eq DMZ Stra form}. This yields:
\begin{equation}\label{eq:chain_rule}
\begin{split}
    d&u(t,x) = d(\exp\left( - \sum_{i=1}^m h^i(x) Y^i_t \right))\sigma(t,x) \\&+\exp\left( - \sum_{i=1}^m h^i(x) Y^i_t \right) d\sigma(t,x)\\
    &= -\sum_{j=1}^mh^j(x)\exp\left( - \sum_{i=1}^m h^i(x) Y^i_t \right)\sigma(t,x)\circ dY_t^j \\
    &\quad + \exp\left( - \sum_{i=1}^m h^i(x) Y^i_t \right)L_0 \sigma(t, x) dt \\
    &\quad + \sum_{i=1}^m h^i(x) \exp\left( - \sum_{i=1}^m h^i(x) Y^i_t \right)\sigma(t, x)\circ dY^i_t \\
    & = \exp\left( - \sum_{i=1}^m h^i(x) Y^i_t \right)L_0\sigma(t, x)dt\\
    & = \left[\exp\left( - \sum_{i=1}^m h^i(x) Y^i_t \right)L_0\exp\left(\sum_{i=1}^m h^i(x) Y^i_t \right)\right] u(t,x)dt.
\end{split}
\end{equation}

Notably, the stochastic terms cancel out, indicating that $u(t,x)$ satisfies a partial differential equation. Our task now is to explicitly calculate the operator 
\[
\exp\left( - \sum_{i=1}^m h^i(x) Y^i_t \right)L_0\exp\left(\sum_{i=1}^m h^i(x) Y^i_t \right).\]

Let $K(t,x) = \sum_{i=1}^m h^i(x) Y^i_t$. Applying the Baker-Campbell-Hausdorff formula, we obtain:
\begin{equation}\label{eq:BCH}
    \begin{split}
        e^{-K(t,x)}L_0e^{K(t,x)}
       &= L_0 + [L_0, K(t,x)] \\&+ \frac{1}{2!}[[L_0, K(t,x)],K(t,x)] \quad + \cdots \\&+ \frac{1}{n!}[\underbrace{[\ldots[L_0, K(t,x)], \ldots K(t,x)]}_{\text{$n$ nested commutators}}] + \cdots
    \end{split}
\end{equation}

Now, consider the order of these terms. We have:
\begin{equation}\label{eq:order}
\begin{split}
    \ord([[L_0, K(t,x)],K(t,x)]) &\leq \ord([L_0, K(t,x)]) - 1\\ &\leq \ord(L_0) -2 = 0.
\end{split}
\end{equation}

From equation \eqref{eq:order}, we can deduce that all terms on the right-hand side of equation \eqref{eq:BCH} are zero starting from the fourth term. Therefore, we can simplify our expression to:\begin{equation}\label{eq:simplified_PDE}
     \frac{\partial u}{\partial t} = \left(L_0 + [L_0, K(t,x)] + \frac{1}{2!}[[L_0, K(t,x)],K(t,x)]\right) u 
\end{equation}

Finally, we substitute $K(t,x) = \sum_{i=1}^m h^i(x) Y^i_t$ back into equation \eqref{eq:simplified_PDE}, which completes the proof.
\end{proof}

The above proposition forms the basis of estimation algebra techniques, due to the method of Wei and Norman \cite{wei1964global}.  In particular, the above equations define a linear PDE with stochastic coefficients, and Wei and Norman show that if the Lie algebra generated by the coefficient operators form a finite-dimensional Lie algebra, the resulting PDE can be integrated explicitly as a product of exponentials.

\begin{definition}\label{def: estimation algebra}
    The estimation algebra $\calE$ is the Lie algebra inside the algebra of differential operators generated by the operators $L_0, L_1, \dots, L_m$, where $L_0 = L^* - \frac{1}{2} h^\top h$ and $L_i$ are multiplication by $h^i$.
\end{definition}

\section{Nonexistence of finite-dimensional estimation algebras}

We are now ready to give the proof of Theorem~\ref{Main Theorem}.  With the preliminaries established in the previous section, it suffices to prove the following theorem.

\begin{theorem}\label{thm:non exist}
    Suppose that $M$ is a closed Riemannian manifold and $h \colon M \rightarrow \R^m$ is smooth and non-constant. Then the estimation algebra $\calE$ is infinite-dimensional.
\end{theorem}

\begin{proof}
    Suppose $h^j$ is non-constant for some $j$.  Let $H_i$ be defined by
    \begin{equation}
        \begin{aligned}
            H_0 &:= h^j\\
            H_{i+1} &:= [[L_0, H_i],H_i] = \langle \grad H_i, \grad H_i \rangle
        \end{aligned}
    \end{equation}
We can define a non-linear operator $Q: f \in C^{\infty}(M)\rightarrow Q(f) =  \langle \grad f, \grad f \rangle \in C^{\infty}(M)$.  We claim that if $h$ is nonconstant and smooth, then the sequence $\{ h, Qh, Q^2 h, \dots \}$ is linearly independent. 

\begin{lemma}\label{lemma:operator_Q}
For the operator $Q$ on a compact Riemannian manifold $M$:
\begin{enumerate}
    \item $Qf = 0$ if and only if $f$ is constant.
    \item There are no functions $f,g$ such that $\langle\grad f, \grad g\rangle = 1$ everywhere on $M$.
\end{enumerate}
\end{lemma}

\begin{proof}[Proof of Lemma \ref{lemma:operator_Q}]
 If $f$ is non-constant, then $df$ is not identically zero. Consequently, $\grad f$ cannot be identically zero, which implies $\langle\grad f, \grad f\rangle$ is not identically zero. Conversely, if $f$ is constant, then $\grad f = 0$, and thus $Qf = 0$. So, we have proved the first one.
    
Since $M$ is compact, every $f \in C^{\infty}(M)$ has at least one critical point. Let $x_0$ be a critical point of $f$, i.e., $\grad f(x_0) = 0$. Then, $0 = \langle\grad f(x_0), \grad g(x_0)\rangle \neq 1$.

\end{proof}
By using Lemma \ref{lemma:operator_Q}, $Q^if$ can not be an identically zero function for any $i\ge 1$ for any non-constant function $f$. We note that for any smooth function $f$, $Qf \ge 0$, is equal to $0$ at critical points of $f$, and $Qf$ must be nonzero.  In particular, the critical points of $f$ are also critical points of $Qf$, so inductively, if $x$ is a critical point of $f$, then $x$ is a critical point of $Q^i(f)$ for all $i$ and $Q^i(f) = 0$ for all $i \ge 1$.  Since $M$ is compact, $Qf$ must have at least one critical point that is not a critical point of $f$ itself, for example, the maximum value of $Qf$.  Let $x_i$ be a sequence so that $Q^i h(x_i) \ne 0$ and $x_i$ is a critical point of $Q^i h$.  By construction $Q^i h(x_i) \ne 0$ and $Q^j h(x_i) = 0$ for all $j > i$.  Thus, we observe that the matrix $A_{ij} = Q^i h(x_j)$ is upper triangular and nonzero on the diagonal, hence nonsingular.  Since each $H_i = Q^i h$ lives in the estimation algebra, this shows that the image of $\mathcal{E}$ under the map $f \mapsto (f(x_1), \dots, f(x_n)) \in \R^n$ spans the image for any $n$, so $\mathcal{E}$ must be infinite-dimensional.

\end{proof}

 \begin{remark}
  In Euclidean space, Ocone showed that any functions $h$ in a finite-dimensional estimation algebra are polynomials of order $\le 2$. If $h$ is a polynomial of degree $\le 2$, then $Q^i h$ is still a polynomial of degree $\le 2$, but the polynomial degree increases arbitrarily for $h$ a polynomial of degree $\ge 3$. This observation was used by Ocone to rule out these polynomial cases in the course of proving that any observation function on Euclidean space must be a polynomial of degree $\le 2$ \cite[Theorem~5.3]{ocone1980topics}.\end{remark}

With a bit more work, the idea above generalizes to Theorem~\ref{Generalization of Main Theorem}.

\begin{theorem}
    Let $M$ be a Riemannian manifold (not necessarily compact) and let $h \colon M \rightarrow \R$ be an observation function such that $h$ has at least two distinct critical points $x$ and $y$ that are connected by a gradient flow along $h$.  Then the estimation algebra generated by $L_0$ and $h$ is infinite-dimensional.
\end{theorem}

\begin{remark}
    Examples of functions $h$ that satisfy the setup above arise in Morse theory when there is at least one nonzero cohomology class in $H^i$ for some $i > 0$, with any coefficients, such as the fundamental class on a compact manifold with $\Z/2$ coefficients.  From a Morse-theoretic perspective, the above theorem can be interpreted as saying that finite-dimensional estimation algebras can only arise in situations where the observation function cannot ``see'' any nontrivial topology.
\end{remark}

\begin{proof}
    Let $A_h(f)$ be the operator $\langle \grad h, \grad f \rangle$.  We claim that the sequence $\{ A_h(h), A_h^2(h), A_h^3(h), \dots \}$ is linearly independent.

    Let $\gamma \colon \R \rightarrow M$ be a parametrization of the gradient flow so that $\gamma(t)$ satisfies the differential equation
    \begin{equation}
        \begin{aligned}
            \gamma'(t) &= (\grad h)(\gamma(t)) \\
            \lim_{t \rightarrow a} \gamma(t) &= x \\
            \lim_{t \rightarrow b} \gamma(t) &= y.
        \end{aligned}
    \end{equation}

    We may assume that there are no critical points along the interior of the path $\gamma(t)$.  Therefore, $A_h(h)(\gamma(t)) > 0$ for all $t$ and approaches $0$ at $\pm \infty$.  Therefore, it must have a maximum.  We also observe that
    \begin{equation}
    \begin{aligned}
        \frac{d}{dt} A_h^n h(\gamma(t)) &= d(A_h^n(h))(\gamma'(t)) \\
        &= \langle \grad A_h^n(h), \grad h \rangle \\
        &= A_h^{n+1}(h)(\gamma(t)).
    \end{aligned}
    \end{equation}
    If the estimation algebra were finite-dimensional, then there would be a linear dependency of the form $\sum_{i=1}^N c_i A_h^{i}(h) = 0$ for sufficiently large $n$ and at least one nonzero coefficients $c_N \ne 0$.  This would imply that $f(t) = A_h(h)(\gamma(t))$ satisfies a nontrivial linear ordinary differential equation in $t$
    \[
    c_0 f(t) + c_1 f'(t) + c_2 f''(t) + \dots + c_N f^{(N)}(t) = 0, \quad \forall t \in \mathbb{R}.
    \]
    with constant coefficients and $c_N \ne 0$.  But every such solution is a linear combination of the form $C_i P_i(t) e^{\lambda_i t}$ for some complex $C_i, \lambda_i$ and polynomials $P_i$.  However, no linear combination can satisfy $\lim_{t \rightarrow \pm \infty} f(t) = 0$ at both $\pm \infty$.  This is a contradiction, so $\mathcal{E}$ cannot be finite-dimensional.
\end{proof}

\section{Conclusion}

In this paper, we have explored the problem of finite-dimensional filtering on closed Riemannian manifolds. Our main result establishes the nonexistence of finite-dimensional estimation algebras for stochastic systems defined on compact manifolds. This finding extends prior work on finite-dimensional filters, particularly in Euclidean spaces. Our conclusion can be extended to general non-compact manifolds under appropriate conditions. Our work demonstrates that a finite-dimensional estimation algebra structure can only exist when the observation function is unable to detect the non-trivial topological structure of the manifold.

\bibliographystyle{IEEEtran}
\bibliography{root}

\end{document}